\documentclass[12pt]{article}
\usepackage{amsfonts}
\usepackage[dvips]{graphicx}
\usepackage{amssymb,color}
\usepackage[latin1]{inputenc}
\usepackage{epic}
\usepackage[latin1]{inputenc}
\usepackage{enumerate}
\usepackage{color}
\usepackage{lettrine}
\usepackage{calc,pifont}
\usepackage[noindentafter,calcwidth]{titlesec}
\usepackage{amsmath}
\usepackage{amssymb}
\usepackage{amsthm}
\usepackage{subfigure}

\def\BBox{\kern  -0.2cm\hbox{\vrule width 0.2cm height 0.2cm}}

\newtheorem{lemma}{Lemma}[section]

\newtheorem{theorem}{Theorem}[section]
\newtheorem{definition}{Definition}[section]

\newtheorem{proposition}{Proposition}[section]
\newtheorem{remark}{Remark}[section]

\begin{document}

\title{On transversal and $2$-packing numbers in straight line systems on $\mathbb{R}^{2}$
}
\author{
G. Araujo-Pardo\thanks{garaujo@math.unam.mx}, L. Montejano\thanks{luis@matem.unam.mx},\\{\small  Instituto de Matem\'aticas}\\ {\small Universidad Nacional
Auton\'oma de M\'exico}\\[1ex]\\A. Montejano\thanks{montejano.a@gmail.com}
\\{\small Unidad Multidisciplinaria de Docencia e Investigación de Juriquilla}%
\\
{\small Universidad Nacional Auton\'oma de M\'exico}\\
{\small Campus Juriquilla, Querétaro.}\\[1ex]\\
A. Vázquez-Ávila\thanks{adrian.vazquez@unaq.edu.mx}\\
{\small Subdirección de Ingeniería y Posgrado}\\
{\small Universidad Aeronáutica en Querétaro}\\
}

\date{}
\maketitle

\begin{abstract}
A \emph{linear system} is a pair $(X,\mathcal{F})$ where
$\mathcal{F}$ is a finite family of subsets on a ground set $X$, and
it satisfies that $|A\cap B|\leq 1$ for every pair of distinct
subsets $A,B \in \mathcal{F}$. As an example of a linear system are
the straight line systems, which family of subsets are straight line
segments on $\mathbb{R}^{2}$. By $\tau$ and $\nu_2$ we denote the size of the minimal transversal and the 2--packing numbers of a linear system respectively. A natural problem is asking about the relationship of these two parameters; it is not difficult to prove that there exists a quadratic function $f$ holding $\tau\leq f(\nu_2)$. However, for straight line system we believe that $\tau\leq\nu_2-1$. In this paper we prove that for any linear system with $2$-packing numbers $\nu_2$ equal to $2, 3$ and $4$, we have that $\tau\leq\nu_2$. Furthermore, we prove that  the linear systems that attains the equality have transversal and $2$-packing numbers equal to $4$, and they are a special family of linear subsystems of the projective plane of order $3$. Using this result we confirm that all straight line systems with $\nu_2\in\{2,3,4\}$ satisfies $\tau\leq\nu_2-1$.
\end{abstract}


$2$--packing,
projective plane.


\section{Introduction}

\label{sec:intro} A \emph{set system} is a pair $(X,\mathcal{F})$ where $%
\mathcal{F}$ is a finite family of subsets on a ground set $X$. A
set system
can be also thought of as a hypergraph, where the elements of $X$ and $\mathcal{%
F}$ are called \emph{vertices} and \emph{hyperedges} respectively.

\begin{definition}
A subset $T\subset X$ is called a \emph{transversal} of
$(X,\mathcal{F})$, if it intersects all the sets of $\mathcal{F}$.
The \emph{transversal number} of $(X,\mathcal{F}),$ denoted by $\tau
(X,\mathcal{F}),$ is the smallest possible cardinality of a
transversal of $(X,\mathcal{F}).$
\end{definition}

Transversal numbers have been studied in the literature in many
diffe\-rent contexts and names. For example with the name of
\emph{piercing number} and
\emph{covering number} (see  \cite{AK06,AK06_2,AKMM01,Eckhoff,Huicochea,Tancer,MS11}).

A system $(X,\mathcal{F})$ is a \emph{$\lambda $-Helly system}, if $\mathcal{%
F}$ satisfies the \emph{$\lambda $-Helly property}, that is, if
every subfamily $\mathcal{F}^{^{\prime }}\subset \mathcal{F}$ has
the property that any $(\lambda +1)$-tuple of $\mathcal{F}^{^{\prime
}}$ is intersecting, then $\mathcal{F}^{^{\prime }}$ is
intersecting. Examples of $\lambda$-Helly systems are families of
convex sets in $\mathbb{R}^{\lambda }$ and the
systems arriving from a $\lambda $-hypergraph as following: Let $G$ be a $%
\lambda $-hypergraph, and consider the set $V(G)$ and the family
$\mathbb{I}$ of maximal independent subset of vertices of $G$ (where
$I\subset V(G)$ is independent, if there is no edge $e \in
E(G)$ such that $e\subset I)$. We associate to the $\lambda $-hypergraph $G$ the following set system $(\mathbb{I},V^{\ast }),$ where $V^{\ast }=\{v^{\ast }\mid v\in
V(G)\}$ and $v^{\ast }=\{S\in \mathbb{I}\mid v\in S\}. $ Then it is
not difficult to see that $\tau (\mathbb{I},V^{\ast })$ is the \emph{%
chromatic number} $\chi (G)$ of $G$. Furthermore, the system $(\mathbb{I}%
,V^{\ast })$ is a $\lambda $-Helly system.

\begin{definition}
A set system $(X,\mathcal{F})$ is called a \emph{linear system}, if
it satisfies $|A\cap B|\leq 1$ for every pair of distinct subsets
$A,B \in \mathcal{F}$.
\end{definition}
Note that any linear system $(X,\mathcal{F})$
is a $2$-Helly system and therefore its transversal number $\tau
(X,\mathcal{F})$ can be
regarded as the chromatic number of the $3$-hypergraph $G$, such that $V(G)=%
\mathcal{F}$ and $\{A,B,C\}\in E(G)$, if and only if, $A\cap B\cap
C=\emptyset $.

\begin{definition}
A subset $R\subseteq\mathcal{F}$ is called a \emph{$2$-packing} of a set system $(X,\mathcal{F})$, if the elements of $R$ are triplewise disjoint. The $2$-packing number of $(X,\mathcal{F})$, denoted by
$\nu_2(X,\mathcal{F})$, is the greatest possible number of a
$2$-packing of $(X,\mathcal{F})$.
\end{definition}

 Note that for a linear system its
$2$-packing number $\nu_2(X,\mathcal{F})$ can be regarded as the
\emph{clique number} $\omega(G)$ of the $3$-hypergraph $G$ described
above. So, for linear systems $(X,\mathcal{F})$ we have:
\begin{equation*}\label{desigualdad}
\lceil \nu_{2}(X,\mathcal{F})/2\rceil \leq \tau
(X,\mathcal{F})\leq\frac{\nu_2(\nu_2-1)}{2},
\end{equation*}
since any maximum $2$--packing of $(X,\mathcal{F})$ induces at most
$\frac{\nu_2(\nu_2-1)}{2}$ double points (points incident to two lines).
In general the transversal number $\tau (X,\mathcal{F})$ of a $\lambda$-Helly system can be arbitrarily large even if $\nu _{\lambda }(X,\mathcal{F}%
)$ is small. There are many interesting works studying the
relationship between $\tau (X,\mathcal{F})$ and $\nu _{\lambda
}(X,\mathcal{F})$, and of course recording the problem of giving a
bound of $\tau (X,\mathcal{F})$ in terms of a function of $\nu
_{2}(X,\mathcal{F})$ (see \cite{AK06}). For linear
systems in a more general context there are bounds to transversal
number \cite{ChD,DF}.

In this paper we denote linear systems by $(P,\mathcal{L})$, where
the elements of $P$ and $\mathcal{L}$ are called \emph{points} and
\emph{lines} respectively. 

We study some specific linear
systems called \emph{straight line systems}, which are defined
below. Some results of this kind of linear systems related with this
work appears in \cite{Tancer}.

\begin{definition}
A \emph{straight line representation} on $\mathbb{R}^{2}$ of a
linear system $(P,\mathcal{L})$ maps each point $x\in P$ to a point
$p(x)$ of $\mathbb{R}^{2}$, and each line $F\in\mathcal{L}$ to a
straight line segment $l(F)$ of $\mathbb{R}^{2}$ in such way that
for each point $x\in P$ and line $F\in\mathcal{L}$ we have $p(x)\in
l(F)$, if and only if, $x\in F$, and for each pair of distinct lines
$F,H\in\mathcal{F}$ we have $l(F)\cap l(H)=\{p(x):x\in F\cap H\}$. A
\emph{straight line system} $(P,\mathcal{L})$ is a linear system,
such that it has a straight line representation on $\mathbb{R}^{2}$.
\end{definition}

The main result of this work is set in the following theorem:

\begin{theorem}\label{thm:main-main}
Let $(P,\mathcal{L})$ be a straight line system with $|\mathcal{L}|>\nu_2(P,\mathcal{L})$. If $\nu_2(P,\mathcal{L})\in\{2,3,4\}$, then
$\tau(P,\mathcal{L})\leq\nu_2(P,\mathcal{L})-1$.
\end{theorem}

We believe that Theorem \ref{thm:main-main} is true in general, that is $\tau(P,\mathcal{L})\leq\nu_2(P,\mathcal{L})-1$,
for $\nu_2(P,\mathcal{L})\geq2$, which seems to be extremely difficult to prove. For the cases where the $2$-packing
number is equal to $2$ or $3$ its proof is easy (see propositions \ref{prop:helly} and \ref{prop:3,2}), and the interesting case is when $\nu_2=4$.

To prove Theorem \ref{thm:main-main} we use the following theorem, which is one of the main results of this work.  

\begin{theorem}\label{thm:main}
Let $(P,\mathcal{L})$ be a linear system with $|\mathcal{L}|>4$. If
$\nu_2(P,\mathcal{L})=4$, then $\tau(P,\mathcal{L})\leq4$. Moreover,
if $\tau(P,\mathcal{L})=\nu_2(P,\mathcal{L})=4$, then
$(P,\mathcal{L})$ is a linear subsystem of $\Pi_3$.
\end{theorem}

It is important to say that this problems is closely related with the Hadwiger-Debrunner $(p,q)$--property for linear set systems $(P,\mathcal{L})$ with $p=\nu_2(P,\mathcal{L})+1$ and $q=3$. A family of sets has the $(p,q)$ property, if among any $p$ members of the family some $q$ have a nonempty intersection. In this contest, our results states that, if $(P,\mathcal{L})$ is a linear system satisfying the $\left(\nu_2(P,\mathcal{L})+1,3\right)$ property, for $\nu_2(P,\mathcal{L})=2,3,4$, then $\tau(P,\mathcal{L})\leq\nu_2(P,\mathcal{L})$.  For more information about the Hadwiger-Debrunner $(p,q)$--property see \cite{B86,B95}.

Theorem \ref{thm:main} states that any linear system $(P,\mathcal{L})$ with $\nu_2(P,\mathcal{L})=4$ and $|\mathcal{L}|>4$ is such that $\tau(P,\mathcal{L})\leq4$, giving a characterization to those linear systems which transversal number is $4$. Furthermore, we prove that these linear systems have not a straight line representation on $\mathbb{R}^{2}$.

It is worth noting that such linear systems $(P,\mathcal{L})$ where $\nu_2(P,\mathcal{L})=\tau(P,\mathcal{L})=4$
are certain linear subsystems of the projective plane of order $3$ (Figure \ref{PP3}).

Recall that a \emph{finite projective plane} (or merely
\emph{projective plane}) is a linear system satisfying that any pair
of points have a common line, any pair of lines have a common point
and there exist four points in general
position (there are not three collinear points). It is well known that, if $%
(P,\mathcal{L})$ is a projective plane then there exists a number $q\in%
\mathbb{N}$, called \emph{order of projective plane}, such that
every point (line, resp.) of $(P,\mathcal{L})$ is incident to
exactly $q+1$ lines (points, resp.), and $(P,\mathcal{L})$ contains
exactly $q^2+q+1$ points (lines, resp.). In addition it is well known that projective planes
of order $q$ exist when $q$ is a power prime. In this work we denote by $%
\Pi_q$ the projective plane of order $q$. For more information about
the existence and the unicity of projective planes see, for
instance, \cite{B86,B95}.

Concerning the transversal number of projective planes it is well
known that every line in $\Pi_q$ is a transversal, then
$\tau(\Pi_q)\leq q+1$. On the other hand $\tau(\Pi_q)\geq q+1$
since a transversal with less than $q$ points cannot exist by a
counting argument (recall that every point in $\Pi_q$ is incident to
exactly $q+1$ lines and the total number of lines is equal to
$q^2+q+1$). Now, related to the $2$-packing number, since
projective planes are dual systems, this parameter coincides with
the cardinality of an \emph{oval}, which is the maximum number of
points in general position (no three of them collinear), and it is
equal to $q+1$ when $q$ is odd  (see for example \cite{B95}).
Consequently, for projective planes $\Pi_q$ of odd order $q$ we
have that $\tau (\Pi_q)=\nu_2(\Pi_q)=q+1$.

In this work we prove, beyond of Theorem \ref{thm:main-main}, if $(P,\mathcal{L})$ is a linear system satisfying
$|\mathcal{L}|>\nu_2(P,\mathcal{L})$ with $\nu_2(P,\mathcal{L})\in\{2,3,4\}$, then
$\tau(P,\mathcal{L})\leq\nu_2(P,\mathcal{L})$; and that every
projective plane $\Pi_q$ of odd order satisfies $\tau
(\Pi_q)=\nu_2(\Pi_q)=q+1$. Furthermore, it is not difficult to prove
that, if $(P,\mathcal{L})$ is a $2$-uniform li\-near system (a simple
graph) with $|\mathcal{L}|>\nu_2(P,\mathcal{L})$, then
$\tau(P,\mathcal{L})\leq\nu_2(P,\mathcal{L})-1$. It is tempting to
conjecture that any linear system $(P,\mathcal{L})$ with
$|\mathcal{L}|>\nu_2(P,\mathcal{L})$ satisfies
$\tau(P,\mathcal{L})\leq\nu_2(P,\mathcal{L})$. Unfortunately that is
not true, in \cite{Eustis} proved, using probabilistic methods the
existence of $k$-uniform linear systems $(P,\mathcal{L})$ for
infinitely many $k$´s and $n=|P|$ large enough, which transversal
number is $\tau(P,\mathcal{L})=n-o(n)$. This $k$-uniform linear
systems has $2$-packing number upper bounded by $\frac{2n}{k}$,
therefore $\nu_2(P,\mathcal{L})<\tau(P,\mathcal{L})$. Moreover, this
implies that $\tau\leq\lambda\nu_2$ does not hold for any positive
$\lambda$.


\section{Results}

Before continuing we give some basic concepts and standard notation \-although many of them can be applied for general set systems.
Let $(P,\mathcal{L})$ be a linear system and $p\in P$ be a point. We use $\mathcal{L}_p$ to denote the set of lines incident to $p$. The \emph{degree} of
$p$ is defined as $deg(p)=|\mathcal{L}_p|$, the maximum degree overall points of the linear systems is denoted by
$\Delta(P,\mathcal{L})$ and the set of points of degree at least $k$ is denoted by $X_k$, this is $X_k=\{p\in P: deg(p)\geq k\}$. A point
of degrees $2$ and $3$ is called \emph{double point} and \emph{triple point} respectively. Finally, a linear system $(P,%
\mathcal{L})$ is called $r$\emph{-regular}, if every point of $P$ has degree $r$, and $(P,\mathcal{L})$ is called $k$\emph{-uniform}, if every line of $%
\mathcal{L}$ has exactly $k$ points.

The following is a trivial observation that will be used later on in
order to avoid annoying cases.

\begin{remark}
\label{rmk:trivial} A linear system $(P,\mathcal{L})$ satisfies $\Delta(P,\mathcal{L})\leq2$, if and only if, $\nu_2(P,\mathcal{L})=|\mathcal{L}|$.
\end{remark}

Note that for linear systems $(P,\mathcal{L})$ with $|\mathcal{L}|>\nu_2(P,\mathcal{L})$ the meaning of
$\nu_2(P,\mathcal{L})=n$ is that, on the one hand there is at least one set of $n$ lines inducing no triple points, and on the other
hand any set of $(n+1)$ lines induces a triple point. In the next propositions \ref{prop:helly} and \ref{prop:3,2} we prove that any linear system $(P,\mathcal{L})$ with $|\mathcal{L}|>\nu_2(P,\mathcal{L})$ is such that $\tau(P,\mathcal{L})\leq\nu_2(P,\mathcal{L})-1$, for
$\nu_2(P,\mathcal{L})=2$ and $\nu_2(P,\mathcal{L})=3$ respectively; consequently, Theorem~\ref{thm:main-main}
holds for $\nu_2(P,\mathcal{L})=2$, and $\nu_2(P,\mathcal{L})=3$. In \cite{Tancer} we proved that, if $(P,\mathcal{L})$ is a straight line
system with the property that, if any $4$ members of $\mathcal{L}$ have a triple point, then $\tau(P,\mathcal{L})\leq2$, that is, if
$(P,\mathcal{L})$ is a straight line systems with $|\mathcal{L}|>4$ and $2\leq\nu_2(P,\mathcal{L})\leq3$, then
$\tau(P,\mathcal{L})\leq2$, which is also a consequence of the propositions \ref{prop:helly} and \ref{prop:3,2} proved below.

\begin{proposition}\label{prop:helly}
If $(P,\mathcal{L})$ is any linear system with $\nu_2(P,\mathcal{L})=2$ and $|\mathcal{L}|>2$, then $\tau(P,\mathcal{L})=1$.
\end{proposition}

\begin{proof}
As any set of three lines has a common point then by $2$--Helly pro\-perty all lines of $\mathcal{L}$ have a common point, that is
$\tau(P,\mathcal{L})=1$.
\end{proof}

It is worth noting that the converse of Proposition~\ref{prop:helly} is also true, that is, any linear system $(P,\mathcal{L})$ with
$\tau(P,\mathcal{L})=1$ satisfies $\nu_2(P,\mathcal{L})=2$.

Next we establish an analogous statement to Proposition \ref{prop:helly} concerning linear systems, which $2$--packing
number is three.

\begin{proposition}\label{prop:3,2}
If $(P,\mathcal{L})$ is any linear system with $\nu_2(P,\mathcal{L})=3$ and $|\mathcal{L}|>3$, then
$\tau(P,\mathcal{L})=2$.
\end{proposition}

\begin{proof}
Recall that $\nu_2(P,\mathcal{L})=3$ implies that any set of four
lines induces a triple point. By Remark~\ref{rmk:trivial},
$\Delta(P,\mathcal{L})\geq 3$, thus the set of points of degree at
least $3$, $X_3$, is not empty. If $|X_3|\geq 2$ we can easily find
a set of four lines inducing no triple point (take two distinct
points in $X_3$, and two lines inciding at each). If $|X_3|=1$, let
$p\in P$ be the only point with $deg(p)\geq 3$. Assume that there is
another point $q\in P$, $q\neq p$, such that $deg(q)=2$, otherwise
$|\mathcal{L}\setminus\mathcal{L}_p|\leq 1$ and the statement holds
true. Now consider $\mathcal{L}''=\mathcal{L}\setminus
(\mathcal{L}_p \cup \mathcal{L}_{q})$. Note that
$\mathcal{L}''=\emptyset$, otherwise we can take four lines (two in
$\mathcal{L}_p$, one in $\mathcal{L}_{q}$ and one more in
$\mathcal{L}''$) inducing no triple point; a contradiction to the
hypothesis $\nu_2(P,\mathcal{L})=3$. Hence, the set $\{p,q\}$ is a
transversal, and $\tau(P,\mathcal{L})=2$ as stated.
\end{proof}

In view of Propositions~\ref{prop:helly} and~\ref{prop:3,2} it is tempting to try to prove that any linear system $(P,\mathcal{L})$
with $\nu_2(P,\mathcal{L})=4$ satisfies $\tau(P,\mathcal{L})\leq 3$. However, as we stated in the introduction the projective plane
$\Pi_3=(P,\mathcal{L})$ of order $3$ (Figure \ref{PP3}) satisfies $\nu_2(\Pi_3)=\tau(\Pi_3)=4$.

\begin{figure}[t]
\begin{center}
\includegraphics[height =4.5cm]{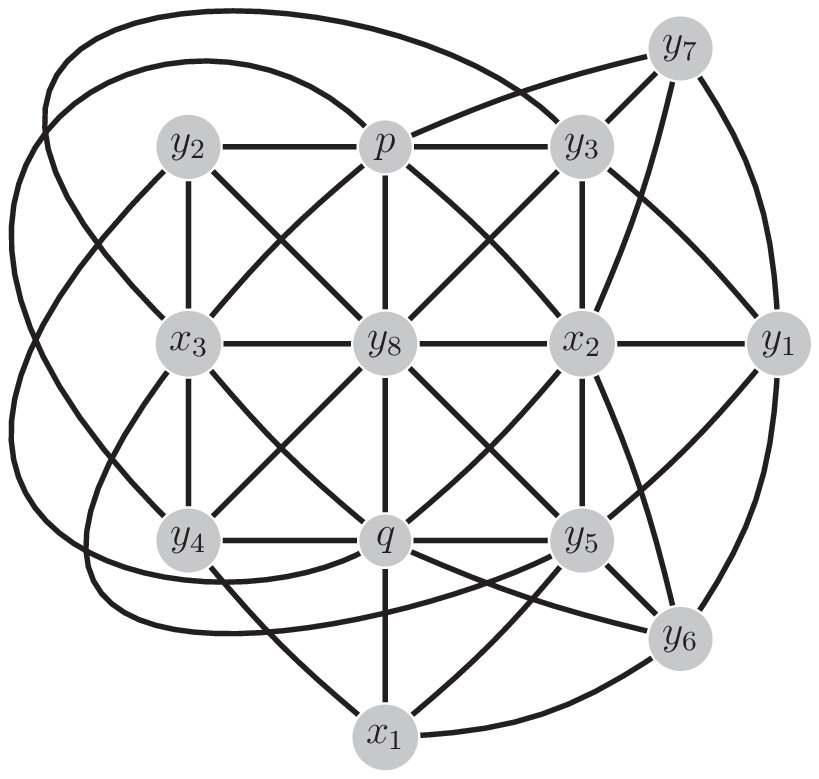}
\end{center}
\caption{} \label{PP3}
\end{figure}

The main work of this paper is to prove that any straight line system $(P,\mathcal{L})$ with $\nu_2(P,\mathcal{L})=4$, and
$|\mathcal{L}|>4$ is such that $\tau(P,\mathcal{L})\leq3$. To prove this we use Theorem \ref{thm:main}, that
is we prove that any linear system $(P,\mathcal{L})$ with $\nu_2(P,\mathcal{L})=4$, and $|\mathcal{L}|>4$ is such that
$\tau(P,\mathcal{L})\leq4$, giving a characterization to those linear systems, which transversal number is $4$, and we prove that
these linear systems have not a straight line representation on $\mathbb{R}^{2}$.

\begin{definition}
A {\emph{linear subsystem}} $(P^{\prime },\mathcal{L}^{\prime })$ of
a linear system $(P,\mathcal{L})$ satisfies that for any line $l^\prime\in\mathcal{L}^\prime$ there exists a line $l\in\mathcal{L}$
such that $l^\prime=l\cap P^\prime$. The \emph{linear subsystem induced} by a set of lines $\mathcal{L}^{\prime}\subseteq \mathcal{L}$ is the linear subsystem
$(P^{\prime },\mathcal{L}^{\prime })$ where
$P^{\prime}=\bigcup_{l\in \mathcal{L}^{\prime }} l$.
\end{definition}

One of the main results of this paper estates the following:

\begin{theorem}\label{thm:main}
Let $(P,\mathcal{L})$ be a linear system with $|\mathcal{L}|>4$. If
$\nu_2(P,\mathcal{L})=4$, then $\tau(P,\mathcal{L})\leq4$. Moreover,
if $\tau(P,\mathcal{L})=\nu_2(P,\mathcal{L})=4$, then
$(P,\mathcal{L})$ is a linear subsystem of $\Pi_3$.
\end{theorem}

In order to prove Theorem \ref{thm:main} we analyze different cases related to the maximum degree of the linear system. Note that
by Remark \ref{rmk:trivial}, a linear system $(P,\mathcal{L})$ satisfying the hypothesis of Theorem~\ref{thm:main} is such that
$\Delta(P,\mathcal{L})>2$. In Lemma \ref{lem:deg5} below we prove that linear systems with $\nu_2(P,\mathcal{L})=4$, and $\Delta(P,\mathcal{L})\geq 5$ are such that $\tau(P,\mathcal{L})\leq3$. The remaining cases, $\Delta(P,\mathcal{L})=3$ and $\Delta(P,\mathcal{L})=4$ are the
cases for which there are linear systems satisfying $\nu_2(P,\mathcal{L})=\tau(P,\mathcal{L})=4$. We handle those cases
in Section~\ref{sec:Delta=3} and Section~\ref{sec:Delta=4} respectively. In each case we describe all linear systems
$(P,\mathcal{L})$ satisfying $\nu_2(P,\mathcal{L})=\tau(P,\mathcal{L})=4$.

Before proceeding to the next section we will prove that linear
systems $(P,\mathcal{L})$ with $\nu_2(P,\mathcal{L})=4$, and
$\Delta(P,\mathcal{L})\geq 5$ are such that $\tau(P,\mathcal{L})=2$,
except for a particular case, which satisfies $\tau(P,\mathcal{L})=3$.

\begin{lemma}
\label{lem:deg5} Any linear system $(P,\mathcal{L})$ with $\nu_2(P,\mathcal{L})=4$, and $\Delta(P,\mathcal{L})\geq 5$ satisfies
$\tau(P,\mathcal{L})\leq3$.
\end{lemma}

\begin{proof}
Recall that $\nu_2(P,\mathcal{L})=4$ implies that any set of five lines induces a triple point. Consider $p\in X_5$, and define
$\mathcal{L}'=\mathcal{L}\setminus \mathcal{L}_p$. Let $(P',\mathcal{L}')$ be the linear subsystem induced by
$\mathcal{L}'$. Note that $|\mathcal{L}'|\geq2$, otherwise $\nu_2(P,\mathcal{L})\leq3$, a contradiction to the hypothesis
$\nu_2(P,\mathcal{L})=4$. If $\mathcal{L}'=\{l_1,l_2\}$, then $\{p,l_1\cap l_2\}$ is a minimum transversal of $(P,\mathcal{L})$, if
$l_1\cap l_2\neq\emptyset$, or else (when $l_1\cap l_2=\emptyset$) the linear system satisfies $\tau(P,\mathcal{L})=3$. On the other
hand, if $|\mathcal{L}'|\geq 3$ we claim that $\nu_2(P',\mathcal{L}')=2$ from which it follows by Proposition~\ref{prop:helly}
that $\tau(P',\mathcal{L}')=1$, therefore  $\tau(P,\mathcal{L})=2$. To verify the claim, suppose on the contrary that there are a set
of three lines $\{l_1, l_2, l_3\}$ of $\mathcal{L}'$ inducing no triple point. This set of three lines induces at most three double points.
By the Pigeonhole Principle there are at least two lines $l, l'\in \mathcal{L}_p$, which do not contain any of these double points,
then the set $\{l,l',l_1, l_2, l_3\}$ induces no triple point; a contradiction to the hypothesis $\nu_2(P,\mathcal{L})=4$.
\end{proof}
\section{The case when $\Delta(P,\mathcal{L})=3$}\label{sec:Delta=3}

We begin this section by introducing some terminology, which will
simplify the description of linear systems $(P,\mathcal{L})$ with $\Delta(P,\mathcal{L})=3$, and
$\nu_2(P,\mathcal{L})=\tau(P,\mathcal{L})=4$.

\begin{definition}
Given a linear system $(P,\mathcal{L})$, and a point $p\in P$, the linear system obtained from $(P,\mathcal{L})$ by
\emph{deleting point $p$} is the linear system $(P^{\prime },\mathcal{L}^{\prime })$ induced by $\mathcal{%
L}^{\prime }=\{l\setminus \{p\}: l\in \mathcal{L}\}$. Given a linear system $%
(P,\mathcal{L})$ and a line $l\in \mathcal{L}$, the linear system
obtained
from $(P,\mathcal{L})$ by \emph{deleting line $l$} is the linear system $%
(P^{\prime },\mathcal{L}^{\prime })$ induced by $\mathcal{L}^{\prime }=%
\mathcal{L}\setminus \{l\}$.
\end{definition}

It is important to state that in the rest of this paper we consider
linear systems $(P,\mathcal{L})$ without points of degree one
because, if $(P,\mathcal{L})$ is a linear system which has all lines
with at least two points of degree $2$ or more, and
$(P',\mathcal{L}')$ is the linear system obtained from
$(P,\mathcal{L})$ by deleting all points of degree one, then they
are essentially the same linear system because it is not difficult
to prove that transversal and $2$-packing numbers of both coincide.

\begin{definition}
Let $(P^{\prime},\mathcal{L}^{\prime })$ and $(P,\mathcal{L})$ be two linear systems. $(P^{\prime},\mathcal{L}^{\prime })$ and $(P,\mathcal{L})$ are isomorphic, and we write $(P^{\prime },\mathcal{L}^{\prime })\simeq(P,\mathcal{L})$, if after deleting vertices of degree 1 or 0 from both, the systems $(P^{\prime},\mathcal{L}^{\prime })$ and $(P,\mathcal{L})$ are isomorphic as hypergraphs.
\end{definition}

\begin{definition}
\label{def:cucarachita} Consider any point $k$, and any line $l$ of
$\Pi_3$, such that $k\not\in l$. We define $\mathcal{C}_{3,4}$ to be
the linear system obtained from $\Pi_3$ by:
\begin{itemize}
\item[i)] deleting point $k$, and its four incident lines,

\item[ii)] deleting line $l$ and its four points.
\end{itemize}
\end{definition}

The linear system $\mathcal{C}_{3,4}=(P_{\mathcal{C}_{3,4}},\mathcal{L}_{%
\mathcal{C}_{3,4}})$ just defined is a $3$-regular and $3$-uniform
linear system with eight points, and eight lines, described as:
\begin{eqnarray*}
P_{\mathcal{C}_{3,4}}&=&\{p,q,x_1,x_2,x_3,y_1,y_3,y_4\}, \\
\mathcal{L}_{\mathcal{C}_{3,4}}&=&\{\{p,y_1,y_3\},\{x_2,x_3,y_1\},%
\{q,y_1,y_4\},\{x_1,x_3,y_4\}, \\
&&\{p,q,x_1\},\{x_1,x_2,y_3\},\{q,x_3,y_3\},\{p,x_2,y_4\}\}.
\end{eqnarray*}
and depicted in Figure \ref{fig:(3,4)cucaracha}. In the next Proposition \ref{prop:(3,4)cucaracha} and Lemma \ref{lem:Delta=3}  we prove that  if $(P,\mathcal{L})$ satisfies $\nu_2(P,\mathcal{L})=4$ and $\Delta(P,\mathcal{L})=3$, then $\tau(P,\mathcal{L})\leq4$; moreover the equality holds only if $(P,\mathcal{L})\simeq\mathcal{C}_{3,4}$.

\begin{figure}[t]
\begin{center}
\includegraphics[height =4.5cm]{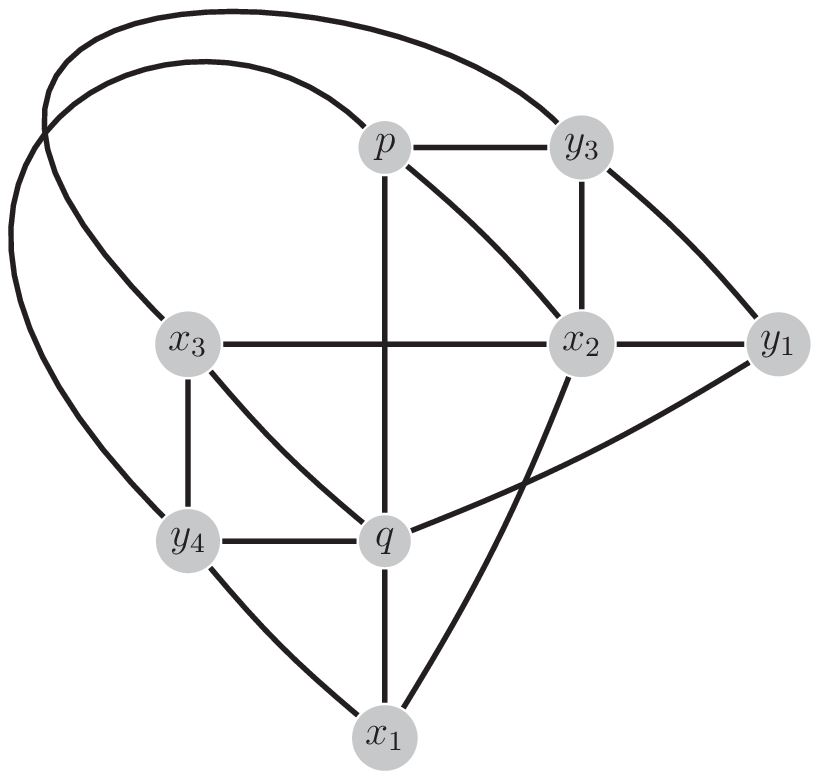}
\end{center}
\caption{} \label{fig:(3,4)cucaracha}
\end{figure}

\begin{proposition}
\label{prop:(3,4)cucaracha} $\nu_2(\mathcal{C}_{3,4})=\tau(\mathcal{C}%
_{3,4})=4$.
\end{proposition}

\begin{proof}
Since the set of lines $$\{\{p,x_2,y_4\},\{q,x_3,y_3\},\{x_1,x_3,y_4\},\{x_1,x_2,y_3\}\}$$
induces no triple point, then $\nu_2(\mathcal{C}_{3,4})\geq4$. On the other hand, it is not difficult
to prove that any set of five lines in $\mathcal{C}_{3,4}$ induces a triple point. Thus $\nu_2(\mathcal{C}_{3,4})=4$.

Since $\{x_1,x_2,y_1,y_4\}$ is a transversal, then $\tau(\mathcal{C}_{3,4})\leq4$. On the other hand, it is easy to
check that  there is no transversal on three points. Thus $\tau(\mathcal{C}_{3,4})=4$.
\end{proof}

\begin{lemma}\label{lem:Delta=3}
Let $(P,\mathcal{L})$ be a linear system with $\nu_2(P,\mathcal{L})=4$, and $\Delta(P,\mathcal{L})=3$. If $(P,\mathcal{L})\not\simeq%
\mathcal{C}_{3,4}$, then $\tau(P,\mathcal{L})\leq3$.
\end{lemma}

\begin{proof}
Let $p$ and $q$ be two points of $P$ such that $deg(p)=3$ and $deg(q)=\max\{deg(x):x\in P\setminus\{p\}\}$. Assume $\deg(q)=3$,
otherwise the statement holds true, since the set of lines $\mathcal{L}\setminus\{l\}$ with $l\in\mathcal{L}_p$ induces no
triple point, and as $|\mathcal{L}\setminus\mathcal{L}_p|\leq2$, then $\tau(P,\mathcal{L})\leq3$ as Lemma \ref{lem:Delta=3} states.
Let $(P'',\mathcal{L}'')$ be the linear subsystem induced by $\mathcal{L}''=\mathcal{L}\setminus(\mathcal{L}_p\cup\mathcal{L}_q)$.
Suppose that $|\mathcal{L}''|\geq3$. We claim that $\nu_2(P'',\mathcal{L}'')=2$ from which it follows by Proposition \ref{prop:helly}
that $\tau(P'',\mathcal{L}'')=1$. Hence Lemma \ref{lem:Delta=3} is proven in this case. To verify the claim
suppose to the contrary that there exists a set of three lines $\{l_1, l_2, l_3\}$ of $\mathcal{L}''$ inducing no triple points.
This set of three lines induces at most three double points $X=\{x_1,x_2,x_3\}$. Since $\Delta(P,\mathcal{L})=3$, by the Pigeonhole Principle there are at least two lines $l_4,l_5\in\mathcal{L}_p\cup\mathcal{L}_q$, which do not contain any point of $X$. Therefore, the set $\{l_1, l_2, l_3,l_4,l_5\}$ induces no triple point in $(P,\mathcal{L})$,
a contradiction to the hypothesis $\nu_2(P,\mathcal{L})=4$. Suppose that $|\mathcal{L}''|\leq2$. Assume that $\mathcal{L}''=\{l_1,l_2\}$
with $l_1\cap l_2=\emptyset$, otherwise the statement holds true. We claim that every line $l\in(\mathcal{L}_p\cup\mathcal{L}_q)\setminus(\mathcal{L}_p\cap\mathcal{L}_q)$ satisfies $l\cap l_1\neq\emptyset$, and $l\cap l_2\neq\emptyset$.
To verify the claim suppose to the contrary that there exists a line $l\in(\mathcal{L}_p\cup\mathcal{L}_q)\setminus(\mathcal{L}_p\cap\mathcal{L}_q)$,
such that $l\cap l_2=\emptyset$. Without loss of generality assume that $l\in\mathcal{L}_p$. By Pigeonhole Principle there are at least two
lines $l_{q_1},l_{q_2}\in\mathcal{L}_{q}$, such that $l\cap l_1\cap l_{q_1} =\emptyset$, and $l\cap l_1\cap l_{q_2} =\emptyset$. Therefore, the set $\{l,l_1,l_2,l_{q_1},l_{q_2}\}$ induces no triple points, a contradiction to the hypothesis $\nu_2(P,\mathcal{L})=4$.
Let $\mathcal{L}_p=\{l_{p_1},l_{p_2},l_{p_3}\}$, and $\mathcal{L}_q=\{l_{q_1},l_{q_2},l_{q_3}\}$.

\textbf{Case $1$:} Suppose that $\mathcal{L}_p\cap\mathcal{L}_q\neq\emptyset$. Let $\{a\}=l_{p_1}\cap l_1$, $\{b\}=l_{p_2}\cap l_1$,
$\{c\}=l_{p_1}\cap l_2$, $\{d\}=l_{p_2}\cap l_2$, where $l_{p_1},l_{p_2}\in\mathcal{L}_p\setminus(\mathcal{L}_p\cap\mathcal{L}_q)$,
then $\{a,d,q\}$ is a transversal, and the statement holds true.

\begin{figure}[t]
  \begin{center}
    \subfigure[]{\includegraphics[height =4cm]{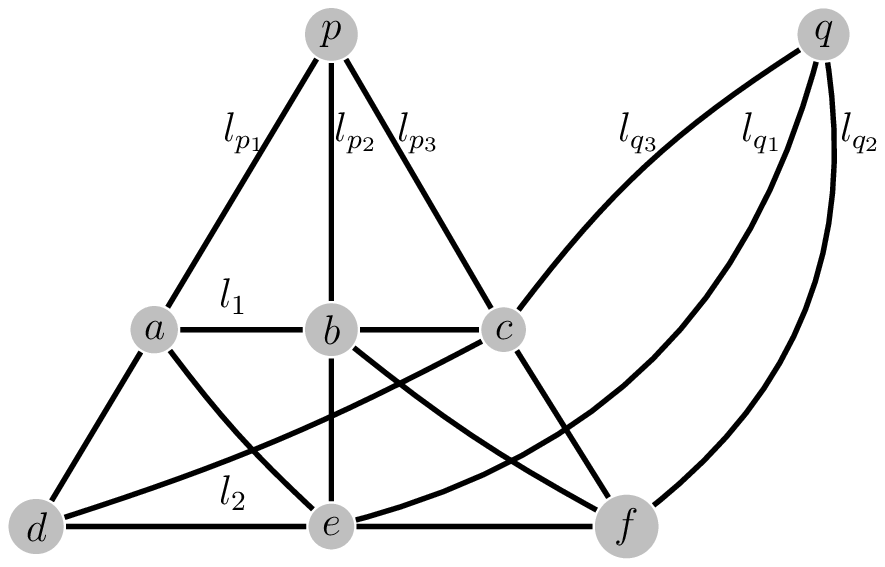}}
    \subfigure[]{\includegraphics[height=4cm]{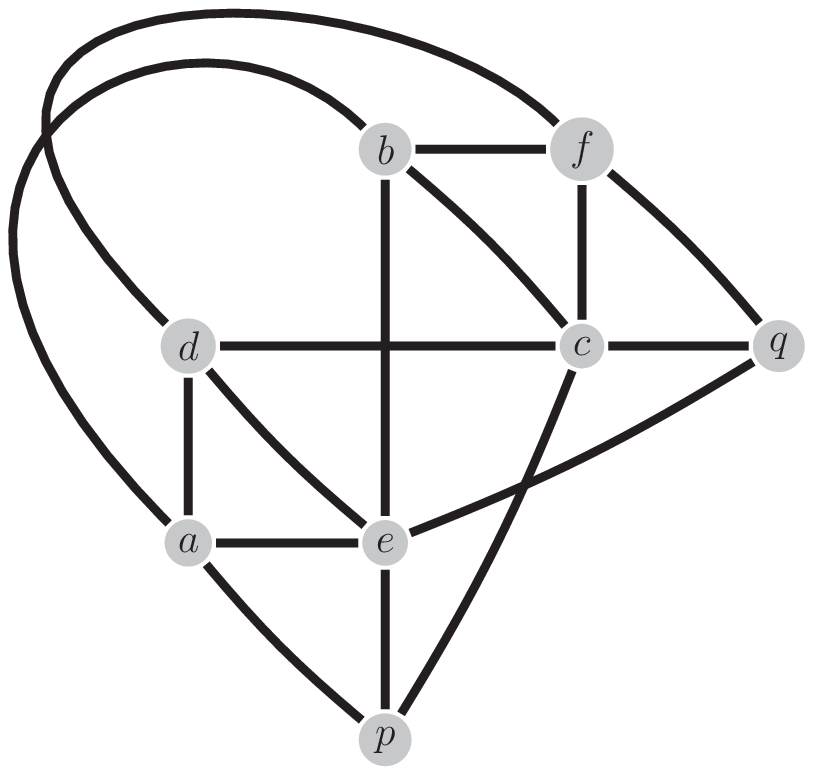}}
  \end{center}
  \caption{}\label{fig:cucarachita}
\end{figure}

\textbf{Case $2$:} Suppose that $\mathcal{L}_p\cap\mathcal{L}_q=\emptyset$. Let $\{a\}=l_{p_1}\cap l_1$, $\{b\}=l_{p_2}\cap l_1$, $\{c\}=l_{p_3}\cap l_1$,
$\{d\}=l_{p_1}\cap l_2$, $\{e\}=l_{p_2}\cap l_2$, and $\{f\}=l_{p_3}\cap l_2$. As $l_{q_i}\cap l_j\neq\emptyset$, for $i=1,2,3$ and $j=1,2$, then given $l_{q_i}\in\mathcal{L}_q$ there exists $l_{p_{s_i}},l_{p_{r_i}}\in\mathcal{L}_p$, $l_{p_{s_i}}\neq l_{p_{r_i}}$, such that
$l_{q_i}\cap l_{p_{r_i}}\cap l_1\neq\emptyset$, and $l_{q_i}\cap l_{p_{s_i}}\cap l_2\neq\emptyset$ (since $l_{q_i}$ induces a triple point on the $2$-packing
$\{l_1,l_2,l_{p_{r_i}},l_{p_{s_i}}\}$, $\{l_1,l_2,l_{p_{s_i}},l_{p_{t_i}}\}$, and $\{l_1,l_2,l_{p_{r_i}},l_{p_{t_i}}\}$, where $\mathcal{L}_p=\{l_{p_{r_i}},l_{p_{s_i}},l_{p_{t_i}}\}$). Let $A_i=\{l_{p_{r_i}},l_{p_{s_i}}\}$ be the set of such lines of $l_{q_i}$. By linearly we have that $|A_i\cap A_j|=1$, for $1\leq i<j\leq3$, and $A_1\cap A_2\cap A_3=\emptyset$, where $A_1, A_2$ and $A_3$ are the corresponding set of lines of
$l_{q_1}, l_{q_2}$ and $l_{q_3}$ respectively. Therefore, either $l_{q_1}\ni a,e$, $l_{q_2}\ni b,f$, and $l_{q_3}\ni d,c$ or $l_{q_1}\ni a,f$,
$l_{q_2}\ni b,d$, and $l_{q_3}\ni c,e$. Without loss of generality assume that $l_{q_1}\ni a,e$, $l_{q_2}\ni b,f$ and $l_{q_3}\ni d,c$
(in the other case we obtain the same linear system, namely the resultant linear systems are isomorphic).

If all three intersections $l_{q_1}\cap l_{p_3}$,  $l_{q_2}\cap l_{p_1}$ and $l_{q_3}\cap l_{p_2}$ are empty, then $(P,\mathcal{L})\simeq\mathcal{C}_{3,4}$, otherwise one of three sets $\{b,d,l_{q_1}\cap l_{p_3}\}$, $\{a,f,l_{q_3}\cap l_{p_2}\}$, $\{c,e,l_{q_2}\cap l_{p_1}\}$ provides a three point transversal. Therefore, the set of points $\{b,d,l_{q_1}\cap l_{p_3}\}$ or $\{a,f,l_{q_3}\cap l_{p_2}\}$ or
$\{c,e,l_{q_2}\cap l_{p_1}\}$ is a transversal of $(P,\mathcal{L})$. Hence, $\tau(P,\mathcal{L})\leq3$ as Lemma \ref{lem:Delta=3} states.
\end{proof}
\section{The case when $\Delta(P,\mathcal{L})=4$}\label{sec:Delta=4}

As in the previous section we begin this section by introducing some terminology to describe linear systems $(P,\mathcal{L})$
with $\Delta(P,\mathcal{L})=4$, and $\nu_2(P,\mathcal{L})=\tau(P,\mathcal{L})=4$.

\begin{definition}
Given a linear system $(P,\mathcal{L})$, we will call a
\emph{triangle} $\mathcal{T}$ of $(P,\mathcal{L})$ as the linear
system induced by three points in general position (non collinear)
and three lines induced by them. 
\end{definition}

\begin{definition}
Consider the projective plane
$\Pi_3$ and a triangle $\mathcal{T}$ of $\Pi_3$. Define
$\mathcal{C}=(P_\mathcal{C},\mathcal{L}_\mathcal{C})$ be the linear
system obtained from $\Pi_3$ by deleting $\mathcal{T}$.
\end{definition}
The linear system
$\mathcal{C}=(P_\mathcal{C},\mathcal{L}_\mathcal{C})$ just defined
has ten points, and ten lines, described as:
\begin{eqnarray*}
P_\mathcal{C}&=&\{p,q,x_1,x_2,x_3,y_1,y_2,y_3,y_4,y_5\}, \\
\mathcal{L}_\mathcal{C}&=&\{\{p,y_1,y_2,y_3\},\{q,y_1,y_4,y_5\},%
\{x_1,x_2,y_3,y_5\}, \{x_1,x_3,y_2,y_4\},\{p,x_2,y_4\}, \\
&&\{p,x_3,y_5\},\{p,q,x_1\},\{q,x_2,y_2\},
\{q,x_3,y_3\},\{x_2,x_3,y_1\}\},
\end{eqnarray*}
and depicted in Figure \ref{fig:PP3}.

\begin{figure}[t]
\begin{center}
\includegraphics[height =4.5cm]{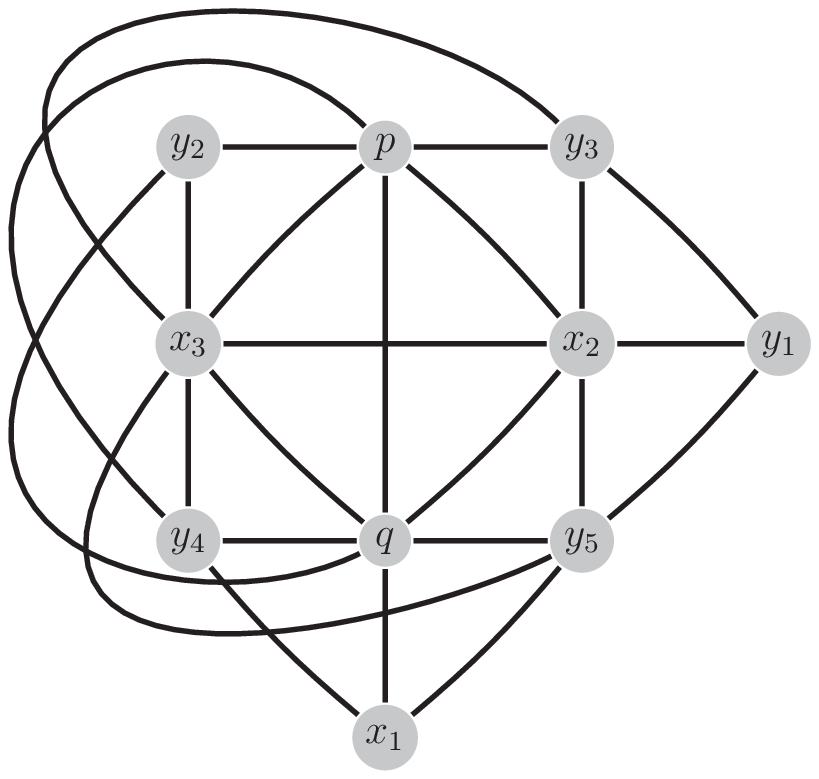}
\end{center}
\caption{} \label{fig:PP3}
\end{figure}

Below we present as a remark some proprieties of $\mathcal{C}$.

\begin{remark}
\label{rmk:propiedadescucaracha}

\

\begin{itemize}
\item $3\leq deg(x)\leq4$, for every $x\in P_\mathcal{C}$,

\item $3\leq|l|\leq4$, for every $l\in\mathcal{L}_\mathcal{C}$,

\item $deg(x)=4$, if and only if, $x$ is adjacent to every $y\in P_\mathcal{C}\setminus\{x\}$,

\item $|l|=4$, if and only if, $l\cap l^{\prime }\neq\emptyset$, for every $l^{\prime }\in \mathcal{L}_\mathcal{C}\setminus\{l\}$,

\item there are no three collinear vertices of degree four,

\item for every $l\in\mathcal{L}_\mathcal{C}$ there exists at most one
line $l^{\prime }\in\mathcal{L}_\mathcal{C}\setminus\{l\}$, such
that $l\cap l^{\prime }=\emptyset$.
\end{itemize}
\end{remark}

\begin{definition}\label{def:cucaracha}
We define $\mathcal{C}_{4,4}$ to be the family of linear systems $(P,\mathcal{L})$ with $\nu_2(P,\mathcal{L})=4$,
such that:

\begin{itemize}
\item[i)] $\mathcal{C}$ is a linear subsystem of $(P,\mathcal{L})$,

\item[ii)] $(P,\mathcal{L})$ is a linear subsystem of $\Pi_3$,
\end{itemize}

this is $\mathcal{C}_{4,4}=\{(P,\mathcal{L}):\mathcal{C}\subseteq(P,\mathcal{%
L})\subseteq\Pi_3\mbox{ and $\nu_2(P,\mathcal{L})=4$}\}$.
\end{definition}

In the next Proposition \ref{prop:igualdadcucaracha} and Lemma \ref{lem:Delta=4} we prove that  if $(P,\mathcal{L})$ satisfies $\nu_2(P,\mathcal{L})=4$ and $\Delta(P,\mathcal{L})=4$, then $\tau(P,\mathcal{L})\leq4$; moreover the equality holds only if $(P,\mathcal{L})\in\mathcal{C}_{4,4}$.

Before continuing we need
some notation for the understand the remainder of this paper. Let $%
(P^{\prime },\mathcal{L}^{\prime })$ be a linear subsystem of a
linear
system $(P,\mathcal{L})$, then we denote $\mathcal{L}\setminus\mathcal{L}%
^{\prime }$ as $\{l\in\mathcal{L}:l^{\prime }\not\subseteq l, l^{\prime }\in%
\mathcal{L}^{\prime }\}$.

\begin{proposition}
\label{prop:igualdadcucaracha}
If $(P,\mathcal{L})\in\mathcal{C}_{4,4}$, then $\tau(P,\mathcal{L})=4$.
\end{proposition}

\begin{proof}

As any line of $(P,\mathcal{L})$ of size four is a transversal of
$(P,\mathcal{L})$ (since any line of size four is a transversal of
$\Pi_3$), then $\tau(P,\mathcal{L})\leq4$. Suppose that $(P,\mathcal{L})$ does not have a transversal of cardinality 4, then there is a transversal $\{a,b,c\}$. Since there are 4 points of degree 4 in $(P_\mathcal{C},\mathcal{L}_C)$, by Pigeonhole Principle at least one of them does not belong $\{a,b,c\}$, denote this point by $x$. Since $|\mathcal{L}_x|=4$, then at least one $l\in\mathcal{L}_x$ is not pierced by $\{a,b,c\}$.
\end{proof}

\begin{lemma}
\label{lem:Delta=4} Let $(P,\mathcal{L})$ be a linear system with $\nu_2(P,\mathcal{L})=\Delta(P,\mathcal{L})=4$,
and $|\mathcal{L}|\geq4$. If $(P,\mathcal{L})\not\in\mathcal{C}_{4,4}$, then $\tau(P,\mathcal{L})\leq3$.
\end{lemma}

\begin{proof}

Recall that $\nu_2(P,\mathcal{L})=4$ implies that any set of five lines induces a triple point. Let $p$ and $q$ be two points of $P$,
such that $deg(p)=4$, and $deg(q)=\max\{deg(x):x\in P\setminus\{p\}\}$. Assume that $deg(q)=4$, otherwise the statement holds true, since if $deg(q)\leq2$ the set of lines $\mathcal{L}\setminus\{l,l'\}$, with $l,l'\in\mathcal{L}_p$, induces no triple point, and as $|\mathcal{L}\setminus\mathcal{L}_p|\leq2$,
then $\tau(P,\mathcal{L})\leq3$. On the other hand, if $deg(q)=3$, then the linear system $(P^\prime,\mathcal{L}^\prime)$ induced by $\mathcal{L}^\prime=\mathcal{L}\setminus\{l_p\}$, with $l_p\in\mathcal{L}_p$, satisfies $\tau(P^\prime,\mathcal{L}^\prime)\leq3$, by Lemma \ref{lem:Delta=3}. Furthermore there exists a transversal $T$ of $(P^\prime,\mathcal{L}^\prime)$ containing the point $p$ (see proof of Lemma \ref{lem:Delta=3}), and therfore $T$ is a transversal of $(P,\mathcal{L})$.

Let $(P'',\mathcal{L}'')$ be the linear subsystem induced by $\mathcal{L}''=\mathcal{L}\setminus(\mathcal{L}_p\cup\mathcal{L}_q)$. Suppose that $|\mathcal{L}''|\leq 2$. Assume that
$\mathcal{L}''=\{l_1,l_2\}$ with $l_1\cap l_2=\emptyset$, otherwise the statement holds true. Proceeding as the proof of Lemma
\ref{lem:Delta=3}, it can be proven that every line $l\in(\mathcal{L}_p\cup\mathcal{L}_q)\setminus(\mathcal{L}_p\cap\mathcal{L}_q)$
satisfies $l\cap l_1\neq\emptyset$, and $l\cap l_2\neq\emptyset$. Without loss of generality assume that there exists a line
$l_q\in\mathcal{L}_q\setminus(\mathcal{L}_p\cap\mathcal{L}_q)$, such that $l_q\cap l_1\cap l_{p_1}\neq\emptyset$ and $l_q\cap l_2\cap
l_{p_2}\neq\emptyset$, where $l_{p_1},l_{p_2}\in\mathcal{L}_p\setminus(\mathcal{L}_p\cap\mathcal{L}_q)$. Then the set $\{l_1,l_2,l_{p_3},l_{p_4},l_q\}$,
where $l_{p_3},l_{p_4}\in\mathcal{L}_p\setminus\{l_{p_1},l_{p_2}\}$, induces no triple point, a contradiction to the hypothesis $\nu_2(P,\mathcal{L})=4$. Suppose that $|\mathcal{L}''|\geq3$. Assume $\nu_2(P'',\mathcal{L}'')\geq3$, otherwise, if $\nu_2(P'',\mathcal{L}'')=2$ from which it follows by Proposition
\ref{prop:helly} that $\tau(P'',\mathcal{L}'')=1$, therefore $\tau(P,\mathcal{L})\leq3$. Let $\{l_1,l_2,l_3\}$ be a set of three
lines of $\mathcal{L}''$ inducing no triple point. This set of three lines induces at most three double points $X=\{x_1,x_2,x_3\}$.
Assume that three lines of $\mathcal{L}_p$, and three lines of $\mathcal{L}_q$ each inside at a point in $X$, otherwise there exist
two lines of $l_4,l_5\in\mathcal{L}_p\cup\mathcal{L}_q$, which do not contain any point of $X$ (by the definition of $deg(q)$), therefore the set of five lines $\{l_1,l_2,l_3,l_4,l_5\}$ induces no triple point, a contradiction to the hypothesis $\nu_2(P,\mathcal{L})=4$.

We claim that there exists one line containing $p$, $q$ and $x$, for some $x\in X$. To verify the claim suppose the contrary. Let $\mathcal{L}_p=\{l_{p_1},l_{p_2},l_{p_3},l_{p_4}\}$, and
$\mathcal{L}_q=\{l_{q_1},l_{q_2},l_{q_3},l_{q_4}\}$, with $(\mathcal{L}_p\setminus\{l_{p_4}\})\cap(\mathcal{L}_q\setminus\{l_{q_4}\})=\emptyset$. Since three lines of $\mathcal{L}_p$, and three lines of $\mathcal{L}_q$ are each incident to a point of $X$, then without loss of generality suppose that
$l_{p_i},l_{q_i}\ni x_i$, for i=1,2,3, and $\{x_{1}\}=l_2\cap l_3$, $\{x_{2}\}=l_3\cap l_1$ and $\{x_{3}\}=l_1\cap l_2$. Then the set
$\{l_1,l_{p_1},l_{p_2},l_{q_1},l_{q_3}\}$ induces no triple point, a contradiction to the hypothesis $\nu_2(P,\mathcal{L})=4$.
Assume that $l_{p,q}\ni x_1$ and $l_{p_i},l_{q_i}\ni x_i$, for $i=2,3$ (see
Figure \ref{fig:pseudocucaracha}), where $l_{p_4}=l_{q_4}=l_{p,q}$. Consider the lines $l_{p_1}$ and $l_{q_1}$, and the following
$2$-packing sets:
\begin{center}
$\mathcal{L}_1=\{l_1,l_2,l_{q_2},l_{pq}\}$,
$\mathcal{L}_2=\{l_1,l_3,l_{q_3},l_{pq}\}$,\\
$\mathcal{L}_3=\{l_1,l_2,l_{p_2},l_{pq}\}$,
$\mathcal{L}_4=\{l_1,l_3,l_{p_3},l_{pq}\}$.
\end{center}

The line $l_{p_1}$ induces a triple point on $\mathcal{L}_1$ and $\mathcal{L}_2$, consequently there must exist intersections
$\{y_2\}=l_2\cap l_{q_2}$ and $\{y_3\}=l_3\cap l_{q_3}$, with $y_2,y_3\in l_{p_1}$, otherwise there exists a set of five lines
$\mathcal{L}_1\cup\{l_{p_1}\}$ or $\mathcal{L}_2\cup\{l_{p_1}\}$ inducing no triple point, a contradiction to the hypothesis
$\nu_2(P,\mathcal{L})=4$. Analogously, the line $l_{q_1}$ induces a triple point on $\mathcal{L}_3$, and $\mathcal{L}_4$. Therefore there
must exist intersections $\{y_4\}=l_2\cap l_{p_2}$, and $\{y_5\}=l_3\cap l_{p_3}$ with $y_4,y_5\in l_{p_2}$. Finally, as
the following set of five lines $\{l_1,l_2,l_3,l_{p_1},l_{q_1}\}$ induces a triple point, there must exists the intersection point
$\{y_1\}=l_1\cap l_{p_1}\cap l_{q_1}$. It is not difficult to prove that the resultant linear system $(P,\mathcal{L})$ (Figure
\ref{fig:cucaracha}) is isomorphic to linear system $\mathcal{C}$. Therefore there exists at least one line $l\in\mathcal{L}\setminus\mathcal{L}_\mathcal{C}$. We claim that each line $l\in \mathcal{L}\setminus\mathcal{L}_\mathcal{C}$ is a line of $\Pi_3$, hence $(P,\mathcal{L})\in\mathcal{C}_{4,4}$,
contradicting the hypothesis $(P,\mathcal{L})\not\in\mathcal{C}_{4,4}$. Before this note that $|\mathcal{L}\setminus\mathcal{L}_\mathcal{C}|\leq3$ (therefore
$|\mathcal{L}|\leq|\mathcal{L}_{\Pi_3}|=13$) since every line of $\mathcal{L}\setminus\mathcal{L}_\mathcal{C}$ induces a triple point
on the $2$-packing $\{l_2,l_3,l_{p_2},l_{p_3}\}$, consequently each line of $\mathcal{L}\setminus\mathcal{L}_\mathcal{C}$ is incident to
one point of $\{x_1,y_4,y_5\}$.

To verify the claim consider the linear system $\mathcal{C}$ depicted in Figure \ref{fig:cucaracha}, and $l$ be a fixed line of
$\mathcal{L}\setminus\mathcal{L}_\mathcal{C}$. We will prove that there exists one line $l'\in\mathcal{L}_{\Pi_3}$, such that $l'=l$.
First we will prove that $l'\subseteq l$. Without losing generality assume $l\ni y_4$ (the same argument is used, if $l\ni x_1$, or
$l\ni y_5$). Line $l$ induces a triple point on the following $2$-packing sets:
\begin{figure}[t!]
  \begin{center}
    \subfigure[]{\includegraphics[height =4.5cm]{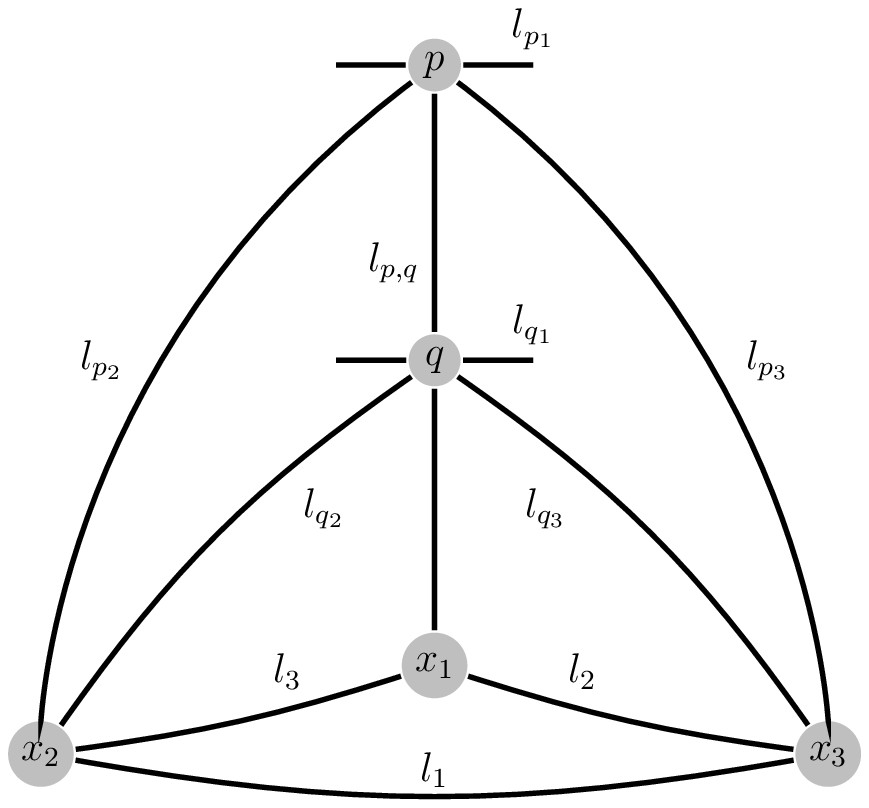}\label{fig:pseudocucaracha}}
    \subfigure[]{\includegraphics[height =4.5cm]{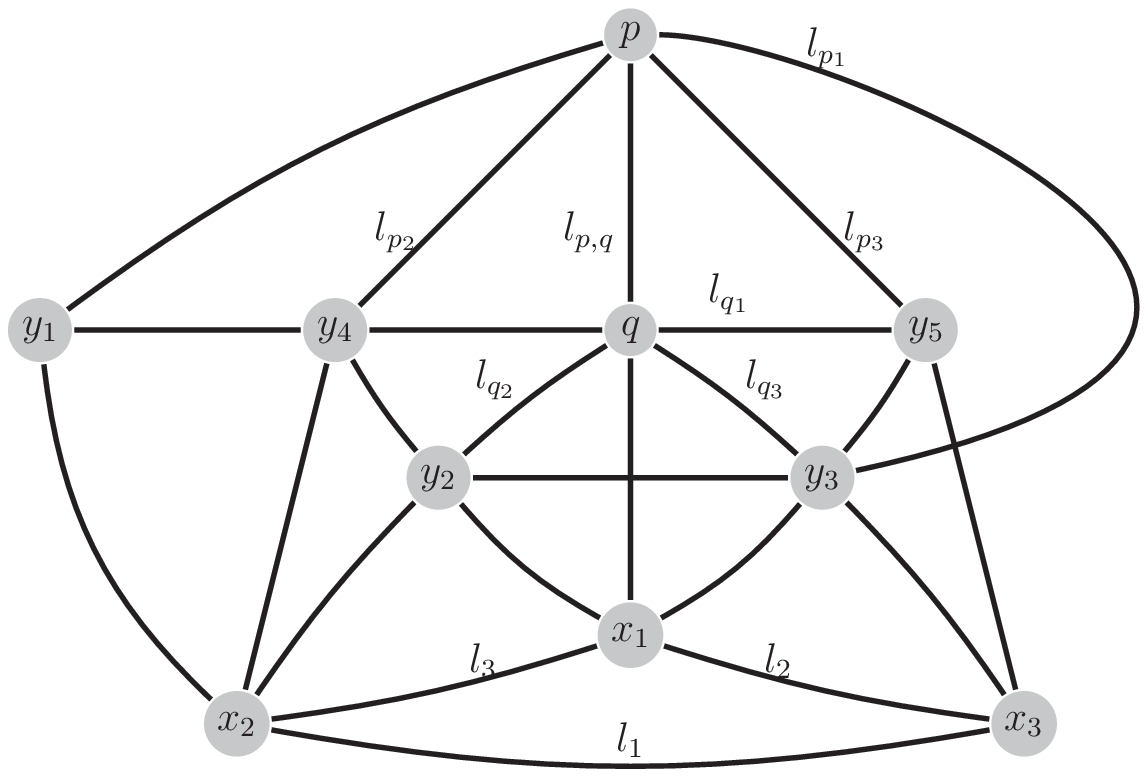}\label{fig:cucaracha}}
  \end{center}
  \caption{}
\end{figure}
$$\mathcal{L}'_1=\{l_1,l_{p_3},l_{q_1},l_{q_2}\},
\mathcal{L}'_2=\{l_1,l_3,l_{q_1},l_{p,q}\},
\mathcal{L}'_3=\{l_3,l_{p_1},l_{q_1},l_{p,q}\},$$
The intersection $\{y_7\}=l_{p_3}\cap l_{q_2}$ and $\{y_8\}=l_1\cap l_{p,q}$ must exists, as well as $y_3,y_7,y_8\in l$ (since
$\{y_3\}\in l\cap l_3\cap l_{p_1}$), otherwise, there must exist a set of five lines $\mathcal{L}'_1\cup\{l\}$, or $\mathcal{L}'_2\cup\{l\}$, or $\mathcal{L}'_3\cup\{l\}$ inducing no triple point, a contradiction to the hypothesis $\nu_2(P,\mathcal{L})=4$. Hence $l'\subseteq l$, where
$l'=\{y_3,y_4,y_7,y_8\}\in\mathcal{L}_{\Pi_3}$ (see Figure \ref{PP3}). To prove that $l\subseteq l'$ is sufficient to verify
that any line $\widetilde{l}$ of $\mathcal{L}\setminus\mathcal{L}_\mathcal{C}$ different of $l$
satisfies $\widetilde{l}\cap l\subseteq l'$, since there are no points of degree one in $l$. Let $\widetilde{l}$ be a line as
before. Without loss of generality assume $y_5\in \widetilde{l}$ (the same argument is used if $l\ni x_1$). Since the line
$\widetilde{l}$ induces a triple point on the $2$-packing $\{l_1,l_3,l_{p_1},l_{p,q}\}$ the intersection $\widetilde{l}\cap
l_1\cap l_{p,q}$ must exist. As $y_8=l\cap l_1\cap l_{p,q}$, then $y_8= \widetilde{l}\cap l'\cap l$, therefore $\widetilde{l}\cap l\in l'$.
\end{proof}

\textbf{Proof of Theorem \ref{thm:main}.} Let $(P,\mathcal{L})$ be a
linear system satisfying the hypothesis of Theorem \ref{thm:main}. If $\Delta(P,%
\mathcal{L})=3$, then by Lemma \ref{lem:Delta=3} we have $\tau(P,\mathcal{L}%
)\leq3$, unless that $(P,\mathcal{L})\simeq \mathcal{C}_{3,4}$ where
by Proposition \ref{prop:(3,4)cucaracha} we have
$\tau(P,\mathcal{L})=4$. On the other hand, if
$\Delta(P,\mathcal{L})=4$, then by Lemma \ref{lem:Delta=4}
we have $\tau(P,\mathcal{L})\leq3$, unless the linear system $(P,%
\mathcal{L})\in\mathcal{C}_{4,4}$ whereby Proposition \ref%
{prop:igualdadcucaracha} we have $\tau(P,\mathcal{L})=4$. Finally, if $%
\Delta(P,\mathcal{L})\geq5$, by Lemma \ref{lem:deg5} we have $\tau(P,%
\mathcal{L})\leq3$. This concludes the proof of Theorem
\ref{thm:main}.\qed

\section{Proof of the Main Theorem}\label{sec:end}

Before continuing with the last part of this paper we need some definitions and results.

\begin{definition}
The \emph{incidence graph} of a set system $(X,\mathcal{F})$, denoted by $B(X,\mathcal{F})$, is a bipartite graph with vertex set
$V=X\cup\mathcal{F}$, where two vertices $x\in X$, and $F\in\mathcal{F}$ are adjacent, if and only if, $x\in F$.
\end{definition}

According to \cite{KKS2008} any straight line system is \emph{Zykov-planar} (see \cite{Zykov}). Zykov proposed to represent the lines of a set system by a subset of the faces of a planar map (map on $\mathbb{R}^{2}$). That is, a set system $(X,\mathcal{F})$ is Zykov-planar, if there exists a planar graph $G$ (not necessarily a simple graph), such that $V(G)=X$, and $G$ can be drawn in the plane
with faces of $G$ two-colored (say red and blue), so that there exists a bijection between the red faces of $G$, and the subsets of
$\mathcal{F}$, such that a point $x$ is incident with a red face, if and only if, it is incident with the corresponding subset. Walsh in
\cite{Walsh} has shown that definition of Zykov is equivalent to the following: A set system $(X,\mathcal{F})$ is Zykov-planar, if and
only if, the incidence graph $B(X,\mathcal{F})$ is planar.

\textbf{Proof of Theorem \ref{thm:main-main}.} By Propositions \ref{prop:helly} and \ref{prop:3,2} we only need to prove the case when $\nu_2=4$.
We consider any linear system $(P,\mathcal{L})$ with $\nu_2(P,\mathcal{L})=4$, and $|\mathcal{L}|>4$. Suppose that
$(P,\mathcal{L})\simeq\mathcal{C}_{3,4}$. We shall prove that $(P,\mathcal{L})$ is not Zykov-planar. Moreover, as $\mathcal{C}_{3,4}$
is a linear subsystem of $\mathcal{C}\in\mathcal{C}_{4,4}$, then any element of $\mathcal{C}_{4,4}$ is not Zykov-planar.
If $(P,\mathcal{L})$ is a straight line system then $(P,\mathcal{L})$ is Zykov-planar, therefore
the incidence graph $B(P,\mathcal{L})$ of $(P,\mathcal{L})$ is a planar graph, but it is not difficult to prove that
$B(P,\mathcal{L})$ is not a planar graph, which is a contradiction. Therefore, there does not exist a straight line representation on $\mathbb{R}^{2}$
of $(P,\mathcal{L})$. On the other hand, if $(P,\mathcal{L})\not\simeq\mathcal{C}_{3,4}$ or
$(P,\mathcal{L})\not\in\mathcal{C}_{4,4}$ with $\nu_2(P,\mathcal{L})=4$, and $|\mathcal{L}|>4$, then by Lemas
\ref{lem:deg5}, \ref{lem:Delta=3}, and \ref{lem:Delta=4} we have $\tau(P,\mathcal{L})\leq3$, as Theorem \ref{thm:main-main} states. \qed

\

{\bf Acknowledgment}

\

The authors thank the referee for many constructive suggestions to improve this paper.


\end{document}